\documentclass[11pt]{article}
\usepackage{geometry}                
\usepackage{graphicx}
\usepackage{amssymb}
\usepackage{epstopdf}
\usepackage{amsthm}
\usepackage{amsmath}
\usepackage{mathtools}
\usepackage{txfonts}

\newtheorem{definition}{Definition}
\newtheorem{theorem}{Theorem}
\newtheorem{lemma}{Lemma}
\newtheorem{corollary}{Corollary}
\newtheorem{conjecture}{Conjecture}

\title{Diophantine Inequalities as a Problem of Difference between Consecutive Primes}
\author{Felix Sidokhine}

\begin{document}

\maketitle

\abstract{
\noindent In the present paper, we have developed a method for solving \textit{diophantine inequalities} using their relationship with the \textit{difference between consecutive primes}.
\newline
\newline
Using this approach we have been able to prove some theorems, including Ingham's exponential theorem as well as some new results. Diophantine inequalities and their connection with Cramer's and Andrica's conjectures are also discussed.
}

\section{Introduction}

In the present paper, we have developed a method for solving \textit{diophantine inequalities} using their relationship with the \textit{difference between consecutive primes}.
\par
\noindent Our approach to Bertrand's postulate and other problems involving primes is based on using \textit{auxiliary inequalities} and \textit{estimates of difference between consecutive primes}. An \textit{auxiliary inequality} for an \textit{initial inequality} is formulated using prime numbers and should be solved over primes. Here is an example plan of proving a proposition:
\par

\begin{description}
\item[Proposition] For any integer $n, n \geq 16$ there is a prime number $q$ such that $\frac{n}{2} < q < n - 4$
\item[Auxiliary proposition] For any prime $p, p \geq 17$ there is a prime $q$ such that $\frac{p+3}{2} < q < p $
\item[Main proposition] If the \textit{auxiliary inequality} is true, then the \textit{initial inequality} is true. 
\item[Final proposition] The \textit{auxiliary inequality} is true.
\end{description}

\par
In the case of the problems studied in this paper, it seems more efficient to look for an \textit{auxiliary inequality} for their \textit{initial inequality}. In many cases, the proofs will use well-known estimates of the \textit{difference between consecutive primes}, most of which can be found either in the original papers \cite{Dusart:1998aa}, \cite{Baker:2001aa} and other facts about primes which can be found in Sierpinski's book \cite{Sierpinski:1964aa}.

\section{Main Results}

Using the approach described above, we have been able to prove the following statements:
\newline

\noindent \textbf{Bertrand's theorem:} For any natural $k$ there exists a computable constant $C(k)$ such that for each integer $n > C(k)$ there is a prime number $p$ such that $\frac{n}{2} < p < n - 2k$
\newline

\noindent \textbf{Fractional theorem:} For any real $k \geq 2$ there is a computable constant $C(k)$ such that for any integer $n > C(k)$ there are at least two primes $p,q$ where $\frac{k-1}{k}n < p,q < \frac{k}{k-1}n$
\newline

\noindent \textbf{Exponential theorem:} For any real $k \geq 2 + \epsilon$, $\epsilon = \frac{2}{19}$, there exists a computable constant $C(k)$ such that for each integer $n > C(k)$, there is a prime number $p$ where $(n-1)^k < p < n^k$
\newline

\noindent \textbf{Theorem} (exponent 3, A.E. Ingham, 1937): There exists a computable constant $C$ such that for each integer $n > C$ there is a prime number $p$ where $(n - 1)^3 < p < n^3$.
\newline

\noindent \textbf{Strong theorem} (exponent 3): There exists a computable constant $C$ such that for each integer       $n > C$ there are at the least two prime numbers $p,q$ where $(n - 1)^3 < p, q < n^3$.
\newline

\noindent \textbf{Brocard's theorem} (exponent 3): There exists a computable constant $C(B)$ such that for each pair primes $p_{n-1}, p_n > C(B)$ there are at the least four prime numbers $p, q, r, s$ where $p_{n-1}^3 < p, q, r, s < p_n^3$.
\newline

\noindent \textbf{Strong Brocard's theorem} (exponent 3): For any natural $k$ there exists such a constant $C(k)$ that for each prime $p_{n-1} > C(k) $ the interval $(p_{n-1}^3, p_n^3)$ contains at the least $2k$ prime numbers. 
\newline

\noindent \textbf{Weak Brocard's theorem} (exponent 2): There exists a computable constant $C(B)$ such that for each pair neighbouring primes $p_{n-1}, p_n > C(B)$ and $(p_n - p_{n-1}) > 3p_n ^\frac{1}{20}$ there are at the least two prime numbers $p, q$ with $p_{n-1}^2 < p, q < p_n^2$. 
\newline

\noindent For the following statements, we will need to define the notion of \textit{Legendre's prime numbers}.

\begin{definition}
A prime $p_n$ is a Legendre prime number if $p_n > a^2 > p_{n-1}$ for some integer $a$. $\Pi_L = \{2,5,11,17,29,37,...\}$, is the set of Legendre primes.
\end{definition}

\noindent \textbf{Legendre's conjecture:} Let for any neighbouring primes $p_{k-1}, p_k$ where $p_k$ is a Legendre prime, the following inequality be true:
\begin{equation}
p_k - p_{k-1} <  2 \sqrt{p_k} + 1
\end{equation}
then for each integer $n$ there is a prime $p$ such that $(n - 1)^2 < p < n^2$.
\newline

\noindent \textbf{Legendre's and Andrica's conjecture:} Let Andrica's conjecture, $\sqrt{p_k} - \sqrt{p_{k-1}} < 1 $, hold for each pair of neighbouring primes $p_{k-1}, p_k$ where $p_k$ is a Legendre prime number, then for each integer $n$ there is a prime $p$ such that $(n - 1)^2 < p < n^2$.
\newline

\noindent \textbf{Oppermann's and the modified Andrica's conjecture:} Let the modified Andrica's conjecture, $\lim_{n \to \infty} (\sqrt{p_{n+1}} - \sqrt{p_n}) = 0$ , be true. Then there exists such a constant $C(O)$ that for any $n > C(O)$ there are at least two primes $p, q$ such that $n^2 < p < n^2 + n$ and $n^2 + n < q < (n + 1)^2$.
\newline

\noindent \textbf{Diophantine inequality and Cramer's conjecture:}  Let $p_n - p_{n-1} = O(\ln^2 p_{n-1})$, hold for each pair of neighbouring primes $p_{n-1}, p_n$ (Cramer's conjecture). Then for any real $\epsilon$ where $0 < \epsilon \leq 1$, there exists such a constant $C(\epsilon)$ such that for each integer $n > C(\epsilon)$ there is a prime number $p$ where $(n - 1)^{1 + \epsilon}  < p < n^{1+\epsilon}$.
 
\section{Proof of the Theorems}

\subsection{Bertrand's theorem}

\begin{theorem}\label{theorem1bt}
For any natural $k$ there exists a computable constant $C(k)$ such that for each integer $n > C(k)$ there is a prime $p$ with $\frac{n}{2} < p < n - 2k$.
\end{theorem}

\begin{theorem}\label{theorem2bt}
For any natural $k$ there exists a computable constant $C(k)$ such that for each prime $p > C(k)$ there is a prime $q$ with $\frac{p + 2k - 1}{2} < q < p$.
\end{theorem}

\begin{lemma}\label{lemma1bt}
If Theorem \ref{theorem2bt} is true, then Theorem \ref{theorem1bt} is true. 
\end{lemma}

\begin{proof}
Let \textit{theorem \ref{theorem1bt}} be false. Then given a fixed $k$, there exists a minimal integer $n_0$ such that the interval $(\frac{n_0}{2},n_0 - 2k)$ contains no primes.
\newline
\newline
Let $n_0 - 2k$ be a composite number and let $p_{n-1},p_n$ be two neighbouring primes such that

\begin{equation}\label{eqn2}
p_{n-1} < n_0 - 2k < p_n 
\end{equation}

\noindent \textit{Inequality \ref{eqn2}} implies that $p_{n-1} < n_0 - 2k \leq p_n-1$ or that $n_0 \leq p_n + 2k - 1$.
\newline
\newline
\noindent The interval $(\frac{p_n + 2k - 1}{2},p_n)$ also contains no prime numbers: this is due to the fact that the intervals $(n_0 - 2k, p_n)$ and $(\frac{p_n + 2k - 1}{2}, n_0 - 2k)$ do not contain any primes, since $\frac{n_0}{2} \leq \frac{p_n + 2k - 1}{2}$
\newline
\newline
Let $n_0 - 2k$ be a prime number. This implies $p_n = n_0 - 2k$. Since the interval $(\frac{n_0}{2}, n_0 - 2k) = (\frac{p_n + 2k}{2},p_n)$ does not contain any prime numbers, hence neither does $(\frac{p_n + 2k - 1}{2},p_n)$ which is a contradiction of theorem \ref{theorem2bt}, assumed to be true.
\end{proof}

\begin{lemma}\label{lemma2bt}
Theorem \ref{theorem2bt} is true if and only if every neighbouring prime numbers $p_{n-1},p_n$ satisfy the following inequality:

\begin{equation}
p_n - p_{n-1} < \frac{p_n -2k + 1}{2}
\end{equation}
\end{lemma}

\begin{proof}
Let theorem \ref{theorem2bt} be true. Then, for any given fixed $k$ there is a prime $q$ such that $q \in  (\frac{p_n + 2k - 1}{2},p_n)$. We can claim $q \leq p_{n-1}$ and therefore $p_{n-1} \in  (\frac{p_n + 2k - 1}{2},p_n)$. Then:

\begin{equation}
p_n - \frac{p_n - 2k + 1}{2} < p_{n-1} < p_n
\end{equation}

\noindent Combining the terms of the inequality, we can obtain:

\begin{equation}
p_n - p_{n-1} < \frac{p_n - 2k + 1}{2}
\end{equation}

\noindent Now let the inequality $p_n - p_{n-1} < \frac{p_n - 2k + 1}{2}$ be true for some fixed given $k$. Then we have:

\begin{equation}
 p_{n-1} > \frac{p_n + 2k - 1}{2}
\end{equation} 

\noindent Therefore $p_{n-1} \in (\frac{p_n + 2k - 1}{2},p_n)$ and therefore \textit{lemma \ref{lemma2bt}} holds.

\end{proof}

\begin{theorem}\label{theorem3bt} 
Bertrand's theorem is true if any pair of neighbouring primes $p_{n-1},p_n$ satisfy the inequality
\begin{equation}
p_n - p_{n-1} < \frac{p_n - 2k + 1}{2}
\end{equation}

\end{theorem}

\begin{proof}
Theorem \ref{theorem3bt} is a consequence of theorems \ref{theorem1bt} and \ref{theorem2bt} and lemmas \ref{lemma1bt} and \ref{lemma2bt}
\end{proof}

\begin{theorem}\label{theorem4bt}
For an integer $k \geq 1$ there is a computable constant $C(k)$ such that for any pair of neighbouring primes $p_{n-1},p_n > C(k)$, the following inequality holds:

\begin{equation}
p_n - p_{n-1} < \frac{p_n - 2k + 1}{2}
\end{equation}
\end{theorem}

\begin{proof}
Using the result from \cite(Dusart) that any $n > 463$, $p_{n+1} \leq p_n(1 + \frac{1}{\ln^2p_n})$, we have the following inequality:

\begin{equation}
p_n - p_{n-1} \leq \frac{p_{n-1}}{\ln^2 p_{n-1}} < \frac{p_n}{\ln^2p_{n-1}}
\end{equation}

\noindent Therefore, we need to find an $n$ such that the following inequality is satisfied;

\begin{equation}
\frac{p_n}{\ln^2p_{n-1}} < \frac{p_n - 2k + 1}{2}
\end{equation}

\noindent Since $\ln(p_n)$ is a strictly increasing function, then there exists such an $n_0$ that for any $n > n_0 + 1$ this inequality takes place; in this particular problem it is sufficient to take $n_0$ in such a way that $p_{n_0-1} < 4k < p_{n_0}$. Therefore the inequality will take place for any $n$ greater than $\max(n_0+1,465)$ 
\end{proof}

\begin{theorem}\label{theorem5bt}
Theorem \ref{theorem1bt} is true for all integers $n > \max(p_r,p_{465})$ where $p_{r-1} < 4k < p_r$.
\end{theorem}

\begin{proof}
Theorem \ref{theorem5bt} is a consequence of theorems \ref{theorem2bt}, \ref{theorem3bt} and \ref{theorem4bt} as well as lemmas \ref{lemma1bt} and \ref{lemma2bt}
\end{proof}

\subsection{Exponential Theorem}

\begin{theorem}\label{theorem1et}
For any fixed real $k \geq \frac{40}{19}$ there exists a constant $C(k)$ such that for each integer $n > C(k)$ there is a prime $q$ where $n - (k - \frac{1}{2})n^{\frac{k-1}{k}} < q < n$
\end{theorem}

\begin{theorem}\label{theorem2et}
For any real number $k \geq \frac{40}{19}$ there exists a natural constant $C(k)$ such that for each prime $p > C(k)$ there is a prime number $q$ where $p - (k - \frac{1}{2})p^{\frac{k-1}{k}} < q < p$
\end{theorem}

\begin{lemma}\label{lemma1et}
Theorem \ref{theorem1et} is true for all integers $n > C(k)$ if and only if theorem \ref{theorem2et} is true for all primes $p > C(k)$
\end{lemma}

\begin{proof}
Let theorem \ref{theorem1et} be true, then theorem \ref{theorem2et} is true for all primes greater than $C(k)$. 
\newline
\newline
Let theorem \ref{theorem2et} be true for all primes $p \geq p_r$ where $p_{r-1} \leq C(k) < p_r$, but assume theorem \ref{theorem1et} is false for some integers. Let $n_0$ be the minimal integer for which theorem \ref{theorem1et} does not hold. This implies that the interval $(n_0 - (k - \frac{1}{2})n_0^\frac{k-1}{k},n_0)$ contains no prime numbers.
\newline
\newline
Let $p_{n-1},p_n$ be a pair of neighbouring primes such that $p_r \leq p_{n-1} < n_0 < p_n$. The the interval $(p_n - (k - \frac{1}{2})p_n^{\frac{k-1}{k}},p_n)$ does not contain any prime numbers. Indeed $(n_0, p_n)$ doesn't have any prime numbers and $(p_n - (k - \frac{1}{2})p_n^\frac{k-1}{k},n_0)$ has no prime numbers since it is a subset of $(n_0 - (k - \frac{1}{2})n_0^\frac{k-1}{k},n_0)$ since $n_0 - (k - \frac{1}{2})n_0^\frac{k-1}{k} < p_n - (k - \frac{1}{2})p_n^\frac{k-1}{k}$. Thus, this allows us to conclude that the interval $(p_n - (k - \frac{1}{2})p_n^\frac{k-1}{k},p_n)$ contains no prime numbers, a contradiction with theorem \ref{theorem2et} assumed to be true.
\end{proof}

\begin{theorem}\label{theorem3et}
The exponential theorem is true if theorem \ref{theorem1et} is true.
\end{theorem}

\begin{proof}
Let theorem \ref{theorem1et} be true, then for any integer $n > C(k)$ there exists a prime number $p$ where $n - (k - \frac{1}{2})n^\frac{k-1}{k} < p < n$. Let $n$ be equal to $n = [m^k]$. Then there is a prime $p \in ([m^k] - (k - \frac{1}{2})[m^k]^\frac{k-1}{k},[m^k])$. Since $([m^k] - (k - \frac{1}{2})[m^k]^\frac{k-1}{k},[m^k]) \subset ( (m-1)^k, m^k)$, then $p \in ( (m-1)^k,m^k)$ and theorem \ref{theorem3et} is true.
\end{proof}

\begin{theorem}\label{theorem4et}
Theorem \ref{theorem2et} is true if and only if there is $C(k)$ such that for every pair neighbouring prime numbers $p_{n-1}, p_n \geq C(k)$ satisfies the following inequality:
\begin{equation} 
p_n - p_{n -1} < (k - 0.5)p_n^\frac{k - 1}{k}
\end{equation}
\end{theorem}

\begin{proof}
Let theorem \ref{theorem2et} be true for all primes $p_n \geq C(k)$, hence there is a  $q \in (p_n - (k - \frac{1}{2})p_n^\frac{k-1}{k},p_n)$ such that $q \leq p_{n-1} < p_n$ and $p_{n-1}$ also belongs to this interval. Thus:

\begin{equation}
p_n - p_{n-1} < p_n - (p_n - (k - 0.5)p_n^\frac{k-1}{k}) = (k - 0.5)p_n^\frac{k-1}{k}
\end{equation}

\noindent Let $p_n - p_{n-1} < (k - \frac{1}{2})p_n^\frac{k-1}{k}$ be true for any pair of neighbouring primes $p_{n-1},p_n > C(k)$, then

\begin{equation}
p_n - (k - 0.5)p_n^\frac{k-1}{k} < p_{n-1} < p_n
\end{equation}

\noindent and $p_{n-1}$ belongs to $(p_n - (k - \frac{1}{2})p_n^\frac{k-1}{k},p_n)$, therefore the interval contains a prime number.

\end{proof}

\begin{theorem}\label{theorem5et}

There exists an integer $C(k)$ such that for every pair of neighbouring primes $p_{n-1},p_n > C(k)$, the following inequality takes place:

\begin{equation}\label{eqn_original}
p_n - p_{n-1} < (k - 0.5)p_n^\frac{k-1}{k}
\end{equation}

\end{theorem}

\begin{proof}
Using the result from \cite{Baker:2001aa}:

\begin{quote}
``Theorem 1. For all $x > x_0$, the interval $[x - x^{0.525}; x]$ contains prime numbers. With enough effort, the value of $x_0$ could be determined effectively."
\end{quote}

\noindent we can claim that any pair of neighbouring primes $p_n,p_{n-1} > x_0$, which in turn implies that:

\begin{equation}
p_n - p_{n-1} < p_n^\frac{21}{40}
\end{equation}

\noindent However, inequality \ref{eqn_original} is always true for $k \geq \frac{40}{19}$. Therefore $C(k) = p_r$ where $p_{r-1} \leq x_0 < p_r$.

\end{proof}

\begin{theorem}\label{theorem6et}
Theorem \ref{theorem1et} is true for all integers $n > C(k) = p_r$ where  $p_{r-1} \leq x_0 < p_r$.
\end{theorem}

\begin{proof}
Theorem \ref{theorem6et} is a consequence of theorems \ref{theorem2et}, \ref{theorem3et}, \ref{theorem4et}, \ref{theorem5et} and lemma \ref{lemma1et}
\end{proof}

\subsubsection{Application of the Exponential Theorem}

\begin{theorem}[exponent 3, A.E. Ingham \cite{Ingham:1937aa}]

For each integer $n > p_r$, where $p_{r-1} \leq x_0 < p_r$ there is a prime number $p$ where:\begin{equation}
(n-1)^3 < p < n^3
\end{equation}

\end{theorem}

\begin{proof}

This theorem is true due to theorems \ref{theorem5et}, \ref{theorem6et} since $k = 3 > \frac{40}{19}$ and for all neighbouring primes $p_{n-1},p_n > p_r$ the following inequality is satisfied:
\begin{equation}
p_n - p_{n-1} < 2.5 p_n^\frac{2}{3}
\end{equation}

\noindent Hence, the exponential theorem is true and for each $n > p_r$ where $p_{r-1} \leq x_0 < p_r$ there is a prime number $p$ with $(n-1)^3 < p < n^3$.
\end{proof}

\noindent Note: we would like to remark that using our approach in the case $p_n - p_{n-1} = O(p_{n-1}^\theta)$ where $\theta = \frac{3}{4} + \epsilon$ is Tchudakoff's constant \cite{Tchudakoff:1936aa} we would not have been able to prove the theorem, however with Ingham's constant $\theta = \frac{5}{8} + \epsilon$ ( \cite{Pintz:2009aa}, \cite{Erdos:1955aa}, \cite{Ingham:1937aa}) we are able to do so.

\begin{theorem}[Quasi - Legendre's theorem]\label{theorem13}
For each integer $n > p_r$ where $p_{r-1} \leq x_0 < p_r$ there is a prime number $p$ such that the following inequality takes place:

\begin{equation}
(n-1)^2 < p^{1-\epsilon} < n^2 \text{ where } \epsilon = 0.05
\end{equation}
\end{theorem}

\begin{proof}
Let us take $k = \frac{40}{19}$, so for any neighbouring pair $p_{n-1},p_n > p_r$ the following is satisfied:

\begin{equation}
p_n - p_{n-1} < \frac{61}{38} p_n^\frac{21}{40}
\end{equation}

\noindent As the exponential theorem is true, for each integer $n > p_r$, where $p_{r-1} \leq x_0 < p_r$ there is a prime number $p$ such that $(n-1)^\frac{40}{19} < p < n^\frac{40}{19}$. This we have $(n-1)^2 < p^{1 - \epsilon} < n^2$ where $\epsilon = 0.05$.
\end{proof}

\begin{lemma}\label{lemma4}
There exists a computable constant $C$ such that as $n$ increases and $n > C$ the number of prime numbers in the interval $((n-1)^3,n^3)$ grows at least as $n^{0.425}$
\end{lemma}

\begin{proof}
According to the Quasi-Legendre's theorem, starting from $C = p_r$ where $p_{r-1} \leq x_0 < p_r$, the intervals $((n - 1)^\frac{40}{19}, n^\frac{40}{19})$ contain at least one prime. Let us estimate the number of such intervals in $((n- 1)^3, n^3)$. The number of such intervals $T$ can be estimated the following way:
\begin{equation}
T = [(n^3)^\frac{19}{40}] - [((n - 1)^3)^\frac{19}{40}] - 1 = [n^\frac{57}{40} - [(n - 1)^\frac{57}{40}] - 1= n^\frac{57}{40} - (n - 1)^\frac{57}{40} + \theta
\end{equation}
\noindent where $|\theta| \leq 3$. $T > n^\frac{17}{40} + \theta = n^{0.425} + \theta$. Since each interval $((n - 1)^\frac{40}{19}, n^\frac{40}{19})$ contain at least one prime, the number of primes increases with $n$ as $n^{0.425}$.
\end{proof}

\subsection{Legendre's Conjecture and Andrica's Conjecture}

\begin{theorem}\label{theoremALC}
There exists a constant $C(L)$ such that for each integer $n > C(L)$ there is a prime number $q$ where $n - 2\sqrt{n} - 1< q < n$.
\end{theorem}

\begin{theorem}\label{theoremBLC}
There exists a constant $C(L)$ such that for each integer $n > C(L)$ there is a prime number $q$ where $p - 2\sqrt{p} - 1< q < p$.
\end{theorem}

\begin{lemma}\label{lemmaALC}
Theorem \ref{theoremALC} is true for all $n > C(L)$ if and only if theorem \ref{theoremBLC} is true for all prime $p \geq p_r$ where $p_{r-1} \leq C(L) < p_r$.
\end{lemma}

\begin{proof}
Let theorem \ref{theoremALC} be true for all integer $n > C(L)$ then theorem \ref{theoremBLC} is true for all prime $p \geq p_r$. Let theorem \ref{theoremBLC} be true for all prime $p \geq p_r$ but theorem \ref{theoremALC} is false some integer $n > p_r$. Let $n_0$ be the minimal integer such that an interval $(n_0 - 2\sqrt{n_0} - 1, n_0)$  contains no prime numbers. 
\newline
\newline
Let  $p_{n-1}, p_n$ be two neighbouring primes such that $p_r \leq p_{n-1} < n_0 < p_n$; then the interval $(p_n - 2\sqrt{p_n} - 1, p_n)$ doesn't contain any prime numbers. Indeed the interval $(n_0, p_n)$ does not contain any prime numbers and the interval $(p_n - 2\sqrt{p_n} - 1, n_0) \subset (n_0 - 2\sqrt{n} - 1, n_0)$, also doesn't contain any prime numbers, since $n_0 - 2\sqrt{n_0} < p_n - 2 \sqrt{p_n}$. Thus $(p_n - 2\sqrt{p_n} - 1, p_n)$ does not contain any primes leading to a contradiction. 
\end{proof}

\begin{lemma}\label{lemmaBLC}
Theorem \ref{theoremBLC} is true for all primes $p \geq p_r$ where $p_{r-1} \leq C(L) < p_r$ if and only if every pair of  neighbouring primes $p_{n-1} \text{ and } p_n \geq p_r$ satisfies the following inequality:
\begin{equation} 
p_n - p_{n-1} < 2\sqrt{p_n}  + 1
\end{equation}
\end{lemma}

\begin{proof}
Let theorem \ref{theoremBLC} be true for all primes $p_n \geq p_r$,  then for $(p_n - 2\sqrt{p_n} - 1, p_n)$ there is $q$ belonging to the interval where $q \leq p_{n-1} < p_n$ and $p_{n -1}$ also belongs to $(p_n - 2\sqrt{p_n} - 1, p_n)$. Thus:
\begin{equation}
p_n - p_{n -1} < p_n - (p_n - 2\sqrt{p_n} - 1) = 2\sqrt{p_n} + 1
\end{equation}

\noindent Let the inequality $p_n - p_{n-1} < 2\sqrt{p_n} + 1$ be true for $p_n \geq p_r$ then $p_n - 2\sqrt{p_n} - 1 < p_{n -1} < p_n$ and $p_{n -1} \in (p_n - 2\sqrt{p_n} - 1, p_n)$. Thus the interval $(p_n - 2\sqrt{p_n} -1, p_n)$ contains a prime number.
\end{proof}

\begin{theorem}\label{theoremCLC}
Legendre's conjecture is true if theorem \ref{theoremALC} is true.
\end{theorem}

\begin{proof}
Let theorem \ref{theoremALC} be true, then for any $n > C(L)$ there is a prime $p$ belonging to the interval $(n - 2\sqrt{n} - 1, n)$. Let $n$ be the square, i.e. $n = m^2$. Then there is a prime $p$ in the interval $(m^2 - 2m - 1, m^2)$. Since $p$ is a prime, it belongs to the interval $((m - 1)^2, m^2)$. As $m^2$ is any positive integer greater than $C$, Legendre's conjecture is true.
\end{proof}

\noindent In general it is possible for Legendre's conjecture to be true while theorem \ref{theoremBLC} is false. Let us introduce the so-called Legendre prime numbers. A prime number $p_n$ is a Legendre prime  if $p_n > a^2 > p_{n-1}$ for some integer $a$. $\Pi_L = \{ l_1=2, l_2=5, l_3=11, l_4=17, l_5=29, l_6=37, l_7=53, l_8=67, l_9=83, l_{10}= 101, l_{11}= 127, ...\}$ is the set of Legendre's prime numbers.

\begin{theorem}\label{theoremDLC}
Legendre's conjecture is true if theorem \ref{theoremBLC} is true for Legendre's prime numbers. 
\end{theorem}

\begin{proof}
Let theorem \ref{theoremBLC} be true for Legendre's prime numbers $l_k \geq l_r \in \Pi_L$ and Legendre's conjecture is false. Let $n_0 > l_r$ be the minimal integer such that an interval $((n_0 - 1)^2, n_0^2)$ doesn't contain any prime numbers. Then the interval $(n_0^2 - 2n_0 - 1, n_0^2)$ doesn't contain any primes. Let us take Legendre's prime number $l_{n_0} > n_0^2$ then the interval $(l_{n_0} - 2\sqrt{l_{n_0}} - 1, l_{n_0})$ doesn't contain any prime numbers. We have the contradiction with the condition of theorem \ref{theoremDLC}, since we assume theorem \ref{theoremDLC} to be true over the set of Legendre primes.
\end{proof}

\noindent Note: Using these results we can give Legendre's conjecture the following algebraic representation:

\begin{corollary}
Let $n$ and $l_n$ belong to the natural numbers and Legendre's prime numbers correspondently. Legendre's conjecture is true if and only if the map $f: n \to l_n$ where $l_{n-1} < n^2 < l_n$ is one-to-one. 
\end{corollary}

\begin{proof}
Let Legendre's conjecture be true, then for each interval $((n - 1)^2, n^2)$ there is a prime number $p \in ((n - 1)^2, n^2)$. Then $p$ is either a  Legendre prime or there is a prime $q \in  ((n - 1)^2, p)$ such that the interval $((n - 1)^2, q)$ doesn't contain any prime numbers and therefore $q$ is  a Legendre prime. Thus we have a map which is strictly increasing so the map is one-to-one. 
\newline
\newline
Let the map $f: n \to l_n$ be one-to-one, then for any $n - 1, n$ there are $p, q$ such that $(n - 1)^2 < p$ and $n^2 < q$ and the intervals $((n -1)^2, p), (n^2, q)$ do not contain any primes. Therefore $p$ belongs to the interval $((n - 1)^2, n^2)$. Thus Legendre's conjecture is true.
\end{proof}

\begin{theorem}\label{theoremELC}
Legendre's conjecture is true if for any pair $p_{n-1}, p_n$ where $p_n$ is Legendre's prime number holds the inequality $p_n - p_{n -1} < 2\sqrt{p_n}  + 1$
\end{theorem}

\begin{proof}
Proof: Theorem \ref{theoremELC} is true due to theorems \ref{theoremALC}, \ref{theoremBLC}, \ref{theoremCLC}, \ref{theoremDLC} and lemmas \ref{lemmaALC}, \ref{lemmaBLC}. 
\end{proof}

\noindent Note: In 1986 year author of paper \cite{Andrica:1986aa} had given the following conjecture: for all pairs neighbouring primes $p_{n-1}, p_n$ takes place the inequality (Andrica's conjecture):

\begin{equation} 
\sqrt{p_n} - \sqrt{p_{n-1}}  < 1.
\end{equation}

\noindent All empirical evidence up to $10^{16}$ confirms that Andrica's conjecture is true. 

\begin{theorem}[Legendre's conjecture]\label{leg_conjecture}
Let Andrica's conjecture, $\sqrt{p_k} - \sqrt{p_{k-1}} < 1$, hold for each pair of neighbouring primes $p_{k-1}, p_k$ where $p_k$ is a Legendre's prime, then for each integer $n$ there is a prime $p$ such that $(n - 1)^2 < p < n^2$.
\end{theorem}

\begin{proof}
Using Andrica's conjecture we have the following inequality:

\begin{equation}
p_n - p_{n -1} < \sqrt{p_n} + \sqrt{p_{n -1}}  < 2\sqrt{p_n}  + 1
\end{equation}

\noindent Therefore due to theorem \ref{theoremELC} under the assumption that the weak Andrica conjecture is true we can conclude that Legendre's conjecture is true.
\end{proof}

\noindent Note: P. Rainboum in his book ``The Little Book of Bigger Primes" writes: ``... Here is another open problem: to show that $\lim_{n \to \infty}(\sqrt{p_{n+1}} - \sqrt{p_n}) = 0$. If true, this would establish (for $n$ sufficiently large) the conjecture of D.Andrica that $\sqrt{p_{n+1}} - \sqrt{p_n} < 1$ for all  $n \geq 1$. In turn, from this inequality, if true, it would follow that between the squares of any two consecutive integers, there is always a prime. This seems indeed true, but has yet to be proved. Note that this is weaker than Opperman's conjecture..." \cite{Ribenboim:2004aa}
\newline
\newline
Note: Legendre's conjecture can be considered without any connection to Andrica's conjecture since one is only a sufficient condition. However if Legendre's conjecture is false then Andrica's conjecture is also false.

\begin{conjecture}
For each pair of neighbouring primes $p_{n-1}, p_n$ where $p_n$ belongs to a set of Legendre's prime numbers $(\Pi_L)$ then:
\begin{equation}
p_n - p_{n -1} < 2\sqrt{p_n}  + 1.
\end{equation}
\end{conjecture}

\begin{conjecture}[Strong Lengendre conjecture]
For any integer $n$ there exist at least two prime numbers $p, q$ where $(n - 1)^2 < p, q < n^2$.
\end{conjecture}

\begin{theorem}
The strong Legendre's conjecture is true if and only if for any $n$, $l_n,l_{n+1} \in \Pi_L$ are not $p_k,p_{k+1} \in \Pi$ for any $k$.
\end{theorem}

\begin{proof}
Let $n$ be an integer and $p,q \in \Pi_L$ be a pair of neighbouring primes. Assume that $p,q$ are not a pair $p_{k},p_{k+1} \in \Pi$. Then $(n-1)^2 < p < n^2 < q$ and there is a prime number $r$ such that $p < r < q$, which implies that the strong Legendre conjecture is true.
\newline
\newline
Let the strong Legendre conjecture be true, then for any integer $n$ there exist two primes $r,s$ such that $(n-1)^2 < r < s < n^2$. But there are also two Legendre primes $p,q$ such that $(n-1)^2 < p \leq r < s < n^2 < q$, therefore the pair $p,q$ is not a pair $p_{k},p_{k+1}$ in $\Pi$.
\end{proof}

\subsection{Oppermann's Conjecture and the Modified Andrica's Conjecture}

\begin{conjecture}[Oppermann's Conjecture]
For any integer $n > 1$ there exist two primes $p, q$ where $n^2 < p < n^2 + n$ and $n^2 + n < q < (n + 1)^2$.
\end{conjecture}

\noindent Although we don't have any instruments for proving Oppermann's conjecture however if we will use modified Andrica's conjecture, $\lim_{n \to \infty}(\sqrt{p_{n+1}} - \sqrt{p_n}) = 0$, then we can prove the following.

\begin{theorem}
There exists such a constant $C(O)$ that for any $n > C(O)$ there are at least two primes $p, q$ such that $n^2 < p < n^2 + n$ and $n^2 + n < q < (n + 1)^2$.
\end{theorem}

\begin{proof}

\begin{theorem}\label{theoremAOC}
There exists an integer $C(O)$ such that for each integer $n > C(O)$ there are two prime numbers $p, q$ such that $n - \sqrt{n}< p < n$ and $n < q < n + \sqrt{n}$.
\end{theorem}

\begin{theorem}\label{theoremBOC}
There exists an integer $C(O)$ such that for each prime $l > C(O)$ there are two prime numbers $p, q$ such that $l - \sqrt{l}< p < l$ and $l < q < l + \sqrt{l}$.
\end{theorem}

\begin{lemma}\label{lemmaAOC}
Theorem \ref{theoremAOC} is true for all integers $n > C(O)$ if and only if theorem \ref{theoremBOC} is true for all primes $p \geq p_r$ where  $p_{r-1} \leq C(O) < p_r$. 
\end{lemma}

\begin{proof}
Let theorem \ref{theoremAOC} be true for all integers $n > C(O)$ then theorem \ref{theoremBOC} is true for all primes $p \geq p_r$. Let theorem \ref{theoremBOC} be true for all prime $p \geq p_r$ but theorem \ref{theoremAOC} is false for some integer $n > p_r$. 
\begin{itemize}
\item Case 1: Let $n_0$ be the minimal integer such that the interval $(n_0 - \sqrt{n_0}, n_0)$ contains no primes. Let $p_{n-1}, p_n$ be two neighbouring prime numbers where $p_r \leq p_{n-1}  < n_0 < p_n$. Then the interval        $(p_n - \sqrt{p_n}, p_n)$ doesn't contain any prime numbers. Indeed the interval $(n_0, p_n)$ doesn't contain any primes and the interval $(p_n - \sqrt{p_n}, n_0) \subset (n_0 - \sqrt{n_0}, n_0)$ doesn't contain any prime numbers since $n_0 - \sqrt{n_0} < p_n - \sqrt{p_n}$ . Thus $(p_n - \sqrt{p_n}, p_n)$ doesn't contain any prime numbers. 

\item Case 2: let $n_0$ be the minimal integer such that an interval $(n_0, n_0 + \sqrt{n_0})$ doesn't contain any primes. Let $p_{m-1}, p_m$ be two neighbouring prime numbers where $p_r \leq p_{m-1} < n_0 < p_m$, then the interval $(p_{m-1}, p_{m-1}+ \sqrt{p_{m-1}})$ doesn't contain any primes. Indeed the interval $(p_{m-1}, n_0)$ doesn't contain any primes and the interval $(n_0, p_{m-1}+ \sqrt{p_{m-1}}) \subset (n_0, n_0+\sqrt{n_0})$ also contains no primes since $p_{m-1}+ \sqrt{p_{m-1}} < n_0 + \sqrt{n_0}$. Thus $(p_{m-1}, p_{m-1} + \sqrt{p_{m-1}})$ doesn't contain any prime numbers leading to a contradiction with the condition of lemma \ref{lemmaAOC}.
\end{itemize}
\end{proof}

\begin{lemma}\label{lemmaBOC}
Theorem \ref{theoremBOC} is true for all primes $p \geq p_r$ where $p_{r-1} \leq C(O) < p_r$ if and only if any pairs of neighbouring prime numbers $p_{n-1}, p_n, p_{m-1}, p_m \geq p_r$ satisfy the following inequalities

\begin{equation}
p_n - p_{n-1} < \sqrt{p_n}
\end{equation}

\begin{equation}
p_m - p_{m -1} < \sqrt{p_{m-1}}
\end{equation} 
 	
\end{lemma}

\begin{proof}
Let theorem \ref{theoremBOC} be true for all primes $p_n \geq p_r$.

\begin{itemize}

\item Case 1: let $p$ belong to the interval $(p_n - \sqrt{p_n}, p_n)$ since $p \leq p_{n-1} < p_n$ then $p_{n-1}$ also belongs to $(p_n - \sqrt{p_n}, p_n)$. Therefore:
\begin{equation}
p_n - p_{n -1} < p_n - (p_n - \sqrt{p_n}) = \sqrt{p_n}
\end{equation}
Let the inequality $p_n - p_{n-1} < \sqrt{p_n}$ be true for every $p_n \geq p_r$, then $p_n - \sqrt{p_n} < p_{n-1} < p_n$ and  $p_{n-1} \in (p_n - \sqrt{p_n}, p_n)$. Thus the interval ($p_n- \sqrt{p_n}, p_n)$ contains a prime number. 

\item Case 2:  let $q$ belong to the interval $(p_{m-1}, p_{m-1} + \sqrt{p_{m-1}})$ since $p_{m-1} < p_m \leq q$ then $p_m$ also belongs to the interval $(p_{m-1}, p_{m-1} + \sqrt{p_{m-1}})$. Hence, $p_m - p_{m-1} < (p_{m-1} + \sqrt{p_{m-1}}) - p_{m-1} = \sqrt{p_{m-1}}$
\newline
\newline
Let the inequality $p_m - p_{m-1} < \sqrt{p_{m-1}}$ be true for each $p_{m-1} \geq p_r$, then $p_{m-1} < p_m < p_{m-1} + \sqrt{p_{m-1}}$ and $p_m \in (p_{m-1}, p_{m-1} + \sqrt{p_{m-1}})$. Therefore, the interval $(p_{m-1}, p_{m-1} + \sqrt{p_{m -1}})$ contains a prime number.
\end{itemize}

\end{proof}

\begin{theorem}\label{theoremCOC}
Theorem \ref{theoremAOC} is true if for any pair $p_{n-1}, p_n \geq p_r$ where $p_{r-1} \leq C(O) < p_r$ satisfies the inequalities:
\begin{equation}
p_n - p_{n-1} < \sqrt{p_n} 	
\end{equation}
\begin{equation}
p_m - p_{m -1} < \sqrt{p_{m-1}}
\end{equation}
\end{theorem}

\begin{proof}
Theorem \ref{theoremCOC} is true due to theorem \ref{theoremBOC} and lemmas \ref{lemmaAOC}, \ref{lemmaBOC}. 
\end{proof}

\begin{theorem}\label{theoremDOC}
The theorem 21 is true if theorem \ref{theoremAOC} is true.
\end{theorem}

\begin{proof}
Let theorem \ref{theoremAOC} be true then for any $n > C(O)$ there are prime numbers $p, q$ which belong to the intervals $(n - \sqrt{n}, n), (n, n + \sqrt{n})$ correspondingly. Let us take $n = (m+1)^2$ for the first bracket and $n = m^2$ for the second bracket. We then have $m^2 + m < p < (m + 1)^2$ and $m^2< q <m^2 + m$ correspondingly.
\end{proof}

\begin{theorem}\label{theoremEOC}
Let the modified Andrica's conjecture be true then there exists such a constant $C(O)$ that for any pairs of neighbouring primes $p_{n-1}, p_n$ and $p_{m-1}, p_m > C(O)$ the following inequalities hold:
\begin{equation}
p_n - p_{n-1} < \sqrt{p_n}
\end{equation}
\begin{equation}
p_m - p_{m -1} < \sqrt{p_{m-1}}
\end{equation}

\end{theorem}

\begin{proof}
Let $d_{m-1} = p_m - p_{m-1}$, and let $S = \{m \in \mathbb{N} | d_{m-1} > \sqrt{p_{m-1}}\}$ be a conflicting set. Let $\#S = \infty$. However, if we evaluate the expression $d_{m-1} > \sqrt{p_{m-1}}$ we will have:
\begin{equation}
\sqrt{p_m} - \sqrt{p_{m-1}} > \frac{\sqrt{p_{m-1}}}{\sqrt{p_m} + \sqrt{p_{m-1}}} > \frac{1}{\sqrt{2} + 1} \text{ for all $m$ in the set $S$}.
\end{equation}
However, this would be a contradiction of the modified Andrica's conjecture, which we assumed to be true (as it would violate Weierstrass' conditions about subsequences), leading us to conclude that $\#S$ was finite.
\end{proof}

\noindent This in turn implies that theorem 21 is true due to theorems \ref{theoremDOC} and \ref{theoremEOC}.
\end{proof}

\begin{corollary}
For any integer $n$ there is a prime $p$ with $n^2 < p < n^2 + n$ if and only if the interval $(n^2, n^2 + n)$ contains Legendre prime number $l_n$.
\end{corollary}

\begin{corollary}
The estimate of the difference between consecutive Legendre prime numbers $l_{n-1}, l_n$ has the following form:
$n < l_{n} - l_{n-1} < 3n - 1$
\end{corollary}

\subsection{Diophantine inequality and Cramer's Conjecture}

``So what Cramer seems to be suggesting, on probabilistic grounds, is that the largest gap between consecutive primes $x$ is $\log^2x$; more precisely,

\begin{equation}
\max_{p_{n-1} < x} (p_n - p_{n-1}) \approx \log^2 x. 				
\end{equation}

\noindent This statement (or the weaker $O(\log^2x)$) is known as `Cramer's Conjecture'." \cite{Granville:1995aa}. Cramer's conjecture has also  gotten some experimental support and the authors of works \cite{Heath-Brown:2005aa}, \cite{Crandall:2005aa}, \cite{Granville:2009aa} believe one can hope that Cramer's conjecture is true.
\newline
\newline
Our goal to show that if Cramer's conjecture is true then the following theorem is true.  

\begin{theorem}[Conditional Theorem]\label{conditional_thm}
For any real $\epsilon$ where $0 < \epsilon \leq 1$ there exists a constant $C(\epsilon)$ such that for each integer $n > C(\epsilon)$ there is a prime number $p$ where $(n -1)^{1+\epsilon} < p < n^{1+\epsilon}$.
\end{theorem}

\begin{proof}

\begin{theorem}\label{theoremADCC}
For any real $\epsilon$, $0 < \epsilon \leq 1$ there exists a constant $C(\epsilon)$ such that for each integer $n > C(\epsilon)$ there is a prime number $q$ such that $n - (0.5 + \epsilon) n^\frac{\epsilon}{1+  \epsilon} < q < n$.
\end{theorem}

\begin{theorem}\label{theoremBDCC}
For any real $\epsilon$, $0 < \epsilon \leq 1$ there exists a constant $C(\epsilon)$ such that for each prime $p > C(\epsilon)$ there is a prime number $q$ such that $p - (0.5 + \epsilon) p^\frac{\epsilon}{1+  \epsilon} < q < p$.
\end{theorem}

\begin{lemma}\label{lemmaADCC}
Theorem \ref{theoremADCC} is true for all integers $n > C(\epsilon)$ if and only if theorem \ref{theoremBDCC} is true for all primes $p > C(\epsilon)$.
\end{lemma}

\begin{proof}
Let theorem \ref{theoremADCC} be true then theorem \ref{theoremBDCC} is also true for all primes $p > C(\epsilon)$. Let theorem \ref{theoremBDCC} be true for all primes $p \geq p_r$ where $p_{r-1} \leq C(\epsilon) < p_r$ but theorem \ref{theoremADCC} is false for some integers. 
\newline
\newline
Let $n_0$ be the minimal rational integer such that the interval $(n_0 - (\frac{1}{2} + \epsilon)n_0^\frac{\epsilon}{1+ \epsilon}, n_0)$ doesn't contain any prime numbers. Let $p_{n-1}, p_n$ be a pair neighbouring primes and $p_r \leq p_{n-1} < n_0 < p_n$ then the interval $(p_n - (\frac{1}{2} + \epsilon) p_n^\frac{\epsilon}{1+ \epsilon}, p_n)$ doesn't contain any prime numbers. 
\newline
\newline
Indeed the interval $(n_0, p_n)$ doesn't contain any prime numbers and $(p_n - (\frac{1}{2} + \epsilon) p_n^\frac{\epsilon}{1+ \epsilon}, n_0) \subset (n_0 - (\frac{1}{2} + \epsilon)n_0^\frac{\epsilon}{1+ \epsilon}, n_0)$ where         $n_0 - (\frac{1}{2} + \epsilon)n_0^\frac{\epsilon}{1+ \epsilon} < p_n - (\frac{1}{2} + \epsilon)p_n^\frac{\epsilon}{1+ \epsilon}$ doesn't contain any prime numbers. Thus $(p_n - (\frac{1}{2} + \epsilon) p_n^\frac{\epsilon}{1+ \epsilon}, p_n)$ doesn't contain any prime numbers. We have a contradiction.
\end{proof}

\begin{theorem}
Let theorem \ref{theoremADCC} be true then the Conditional Theorem is true.
\end{theorem}
 
\begin{proof} 
Let theorem \ref{theoremADCC} be true then for any integer $n > C(\epsilon)$ there is a prime number $p$ with $n - (\frac{1}{2} +\epsilon)n^\frac{\epsilon}{1+  \epsilon} < p < n$. Let an integer $n$ be equal to $n = [m^{1+  \epsilon}]$ where $m^{1+  \epsilon} - 1 < n \leq m^{1+  \epsilon}$. Then there is a prime $q$ belonging to $([m^{1+ \epsilon}] - (\frac{1}{2} + \epsilon)[m^{1+ \epsilon}]^\frac{\epsilon}{1+ \epsilon}, [m^{1+ \epsilon}])$. Since $([m^{1+ \epsilon}] - (\frac{1}{2} + \epsilon) [m^{1+ \epsilon}]^\frac{\epsilon}{1+  \epsilon}, [m^{1+ \epsilon}]) \subset ((m - 1)^{1+\epsilon}, m^{1+\epsilon})$ so $q \in ((m-1)^{1+\epsilon}, m^{1+\epsilon})$. Thus theorem 30 is true.
\end{proof}
 
 \begin{theorem}\label{theoremDDCC} 
Theorem \ref{theoremBDCC} is true if and only if there exists a constant $C(\epsilon)$ such that every pair of neighbouring prime numbers $p_{n-1}, p_n \geq C(\epsilon)$ satisfies the following inequality
\begin{equation}
p_n - p_{n -1} < (0.5 + \epsilon) p_n^\frac{\epsilon}{1+\epsilon}
\end{equation}
\end{theorem} 

\begin{proof} 
Let theorem \ref{theoremBDCC} be true for all primes $p_n \geq C(\epsilon)$, so for the interval $(p_n - (\frac{1}{2} + \epsilon) p_n^\frac{\epsilon}{1+\epsilon}, p_n)$ there is $q$ belonging to this interval and $q \leq p_{n-1} < p_n$ and $p_{n-1}$ also belongs to this interval. Thus:
\begin{equation} 
p_n - p_{n -1} < p_n - (p_n - (0.5 + \epsilon) p_n^\frac{\epsilon}{1+\epsilon}) = (0.5 + \epsilon) p_n^\frac{\epsilon}{1+\epsilon}.
\end{equation}
Let $p_n - p_{n-1} < (\frac{1}{2} + \epsilon) p_n^\frac{\epsilon}{1+\epsilon}$ be true for any pair neighbouring primes $p_{n-1}, p_n > C(\epsilon)$ then $p_n - (\frac{1}{2} + \epsilon) p_n^\frac{\epsilon}{1+\epsilon} < p_{n-1} < p_n$  and $p_{n-1}$ belongs to $(p_n - (\frac{1}{2} + \epsilon) p_n^\frac{\epsilon}{1+\epsilon}, p_n)$. Thus this interval contains a prime number.
\end{proof}


\begin{theorem}\label{theoremEDCC} 
For any real $\epsilon$ where $0 < \epsilon \leq 1$ there exists an integer $C(\epsilon)$ such that for each for every pair of neighbouring primes $p_{n-1}, p_n > C(\epsilon)$ the following inequality takes place:
\begin{equation}
p_n - p_{n -1} < (0.5 + \epsilon) p_n^\frac{\epsilon}{1+\epsilon}.
\end{equation}
\end{theorem}

\begin{proof}
Let Cramer's conjecture \cite{Cramer:1936aa} be true then there exists such a constant $C$ that for all prime numbers $p_n, p_{n -1}$ the following inequality holds: $p_n - p_{n -1} < C \ln^2 p_{n-1}$. Furthermore there exists such $n_0$ that for all $n > n_0 + 1$ the following inequality takes place:
\begin{equation}
p_n - p_{n -1} < C \ln^2 p_{n-1} < (0.5+ \epsilon) p_{n-1}^\frac{\epsilon}{1+\epsilon} < (0.5 + \epsilon) p_n^\frac{\epsilon}{1+\epsilon}.
\end{equation}
\noindent Thus we can take $C(\epsilon)$ as equal to $p_{n_0}$. 
\end{proof}
\noindent Therefore, the conditional theorem is true for all integers $n > C(\epsilon) = p_{n_0}$ due to theorems \ref{theoremADCC}, 30.
\end{proof}

\begin{conjecture}
For any real $\epsilon$ where $0 < \epsilon \leq 1$ there exists such an integer $C(\epsilon)$ that for each integer $n > C(\epsilon)$ there is a prime number $p$ with $(n -1)^{1+ \epsilon} < p < n^{1+  \epsilon}$.
\end{conjecture}

\noindent Note: Given conjecture can consider without any connection to Cramer's conjecture since one is only a sufficient condition. However if given conjecture is false even though for one value $\epsilon$ then Cramer's conjecture also is false.

\begin{theorem}
The modified Andrica's conjecture, $\lim_{n \to \infty}(\sqrt{p_{n+1}} - \sqrt{p_n})=0$ is true if the weak Cramer conjecture, $p_n- p_{n-1} = O(\ln^2p_{n-1})$, is true.
\end{theorem}

\begin{proof}
According to the weak Cramer's conjecture there exists such a constant C that for all prime numbers $p_n, p_{n -1}$, the following inequality takes place:
\begin{equation}
p_n - p_{n -1} < C\ln^2 p_{n-1}
\end{equation}
Furthermore, we get the following inequality:
\begin{equation}
\sqrt{p_n} - \sqrt{p_{n-1}}  < C\frac{\ln^2 p_{n-1}}{\sqrt{p_n} + \sqrt{p_{n -1}}}  < C\frac{\ln^2 p_n}{p_n}
\end{equation} 
and $\lim_{n \to \infty} \sqrt{p_{n+1}} - \sqrt{p_n} = 0$. Thus modified Andrica's conjecture is true.
\end{proof}

\noindent Note: Given theorem can consider without any connection to Cramer's conjecture since one is only a sufficient condition. However if modified Andrica's conjecture is false then Cramer's conjecture also is false.

\subsection{Fractional theorem}

\begin{theorem}[Fractional theorem]\label{theoremFT}
For any real $k \geq 2$ there exists a computable constant $C(k)$ such that for each integer $n>C(k)$ there are two primes $p, q$ such that $\frac{k - 1}{k}n< p, q<\frac{k}{k - 1}n$. 
\end{theorem}

\begin{theorem}\label{theoremAFT}
For any real $k \geq 2$ there exists a computable constant $C(k)$ such that for each integer $n > C(k)$ there is a prime number $p$ where $n < p < \frac{k}{k - 1}n$.
\end{theorem}

\begin{theorem}\label{theoremAprimeFT}
For any real $k \geq 2$ there exists a computable constant $C(k)$ such that for each integer $n > C(k)$ there is a prime number $p$ where $\frac{k - 1}{k}n < p < n$.
\end{theorem}

\begin{theorem}\label{theoremBFT}
For any real $k \geq 2$ there exists a computable constant $C(k)$ such that for each prime $n > C(k)$ there is a prime number $p$ where $n < p < \frac{k}{k - 1}n$. 
\end{theorem}

\begin{theorem}\label{theoremBprimeFT}
For any real $k \geq 2$ there exists a computable constant $C(k)$ such that for each prime $n > C(k)$ there is a prime number $p$ where $\frac{k - 1}{k}n < p < n$.
\end{theorem}

\begin{lemma}\label{lemmaAFT}
Theorem \ref{theoremAFT} is true for all integers $n > C(k)$ if and only if theorem \ref{theoremBFT} also is true for all primes $n > C(k)$. 
\end{lemma}

\begin{proof}
Let theorem \ref{theoremAFT} be true then theorem \ref{theoremBFT} is true for all prime $n > C(k)$. Let theorem \ref{theoremBFT} be true for all primes $p \geq p_r$ where $p_{r-1} \leq C(k) < p_r$ but theorem \ref{theoremAFT} is false for some integers. Let $n_0$ be the minimal natural number such that an interval $(n_0, \frac{k}{k-1}n_0)$ contains no prime numbers. Let $p_{n-1}, p_n$ be a pair neighbouring primes such that $C(k)< p_r\leq p_{n-1}<n_0<p_n$ then the interval $(p_{n-1}, (\frac{k}{k-1})p_{n-1})$ doesn't have any prime numbers. Indeed the interval $(p_{n-1}, n_0)$ doesn't contain any prime numbers nor does the interval $(n_0, (\frac{k}{k-1})p_{n-1}) \subset (n_0, (\frac{k}{k-1})n_0)$. Thus $(p_{n-1}, \frac{k}{k -1}p_{n-1})$ doesn't contain any prime numbers which is a contradiction.
\end{proof} 

\begin{theorem}\label{theoremCFT}
Theorem \ref{theoremBFT} is true if and only if there is such a constant $C(k)$ that for every pair of neighbouring prime numbers $p_{n-1}, p_n > C(k)$ the following inequality is satisfied:
\begin{equation}
p_n - p_{n - 1} < \frac{p_{n-1}}{k - 1}
\end{equation}
\end{theorem}

\begin{proof}
Let theorem \ref{theoremBFT} be true, then the interval $(p_{n-1}, (\frac{k}{k-1}) p_{n-1})$ contains a prime number $q$ such that $p_{n-1}< p_n \leq q$ and $p_n$ belongs to this interval. Thus we have:
\begin{equation}
p_{n - 1} < p_n < (\frac{k}{k-1})p_{n-1}  \text{ and }  p_n - p_{n - 1} < \frac{p_{n - 1}}{k - 1}
\end{equation}
Let the inequality $p_n - p_{n - 1} < \frac{p_{n - 1}}{k - 1}$ be true for any pair prime number then: $p_{n-1} < p_n < p_{n - 1} +  \frac{p_{n - 1}}{k - 1}$ and $p_{n}$ belongs to the interval $(p_{n-1}, (\frac{k}{k-1})p_{n-1})$. Thus any interval $(p_{n-1}, (\frac{k}{k-1}) p_{n-1})$ contains a prime number.
\end{proof}

\begin{theorem}\label{theoremDFT}
Theorem \ref{theoremAFT} is true if and only if for any pair neighbouring primes $p_{n-1}, p_n> C(k)$ the following inequality holds:
\begin{equation}
p_n - p_{n -1} < \frac{p_{n-1}}{k - 1}
\end{equation}
\end{theorem}

\begin{proof}
Theorem \ref{theoremDFT} is true due to theorems \ref{theoremBFT}, \ref{theoremCFT} and lemma \ref{lemmaAFT}.
\end{proof}

\begin{lemma}\label{lemmaAprimeFT}
Theorem \ref{theoremAprimeFT} is true for all integers $n > C(k)$ if and only if theorem \ref{theoremBprimeFT} is true for all prime $n > C(k)$. 
\end{lemma}

\begin{proof}
Let theorem \ref{theoremAprimeFT} be true then theorem \ref{theoremBprimeFT} is true for all prime $n > C(k)$. Let theorem \ref{theoremBprimeFT} be true for all primes $p \geq p_r$ where p$_{r-1}\leq C(k)< p_r$ but theorem \ref{theoremAprimeFT} is false for some integers. 
\newline
\newline
Let $n_0$ be the minimal natural number such that an interval $(\frac{k - 1}{k}n_0, n_0)$ doesn't contain any prime numbers. Let $p_{n-1}, p_n$ be a pair of neighbouring primes $C(k)< p_r\leq p_{n-1}<n_0<p_n$ then the interval $(\frac{k - 1}{k} p_n, p_n)$ doesn't have any prime numbers. Indeed the interval $(n_0, p_n)$ doesn't contain any prime numbers and the interval $(\frac{k - 1}{k}p_n, n_0) \subset (\frac{k - 1}{k}n_0, n_0)$ doesn't contain any prime numbers. Therefore, $(\frac{k - 1}{k} p_n, p_n)$ doesn't contain any prime numbers, leading to a contradiction.
\end{proof}

\begin{theorem}\label{theoremCprimeFT}
Theorem \ref{theoremBprimeFT} is true if and only if there is such a constant $C(k)$ that for every pair of neighbouring primes $p_{n-1}, p_n > C(k)$ the following inequality holds:
\begin{equation}
p_n - p_{n - 1} < \frac{p_n}{k}
\end{equation}
\end{theorem}

\begin{proof}
Let theorem  \ref{theoremBprimeFT} be true, then the interval $q \in (\frac{k - 1}{k} p_n, p_n)$ then $q \leq p_{n-1} < p_n$ and $p_{n-1}$ also belongs to the interval $(\frac{k - 1}{k} p_n, p_n)$. Thus we have
\begin{equation}
p_n - \frac{p_n}{k} < p_{n-1} < p_n \text{ and } p_n - p_{n - 1} < \frac{p_n}{k}
\end{equation}
Let the inequality $p_n - p_{n - 1} < \frac{p_n}{k}$ be true then:
\begin{equation}
\frac{k - 1}{k} p_n = p_n - \frac{p_n}{k} < p_{n-1} < p_n
\end{equation}
and $p_{n-1} \in (\frac{k - 1}{k} p_n, p_n)$. Therefore, the interval $(\frac{k - 1}{k} p_n, p_n)$ contains a prime.
\end{proof}

\begin{theorem}\label{theoremDprimeFT}
Theorem \ref{theoremAprimeFT} is true if and only if for any pair of neighbouring primes $p_{n-1}, p_n> C(k)$ satisfy the inequality:
\begin{equation}
p_n - p_{n -1} < \frac{p_n}{k}
\end{equation}
\end{theorem}

\begin{proof}
Theorem \ref{theoremDprimeFT} is true due to theorems \ref{theoremBprimeFT}, \ref{theoremCprimeFT} and lemma \ref{lemmaAprimeFT}.
\end{proof}

\begin{theorem}\label{theoremEFT}
The fractional theorem is true if there is such an integer $C$ that every pair of neighbouring primes $p_{n-1}, p_n> C$ satisfies the following inequalities:
\begin{equation}
p_n - p_{n -1} < \min(\frac{p_{n-1}}{k - 1}, \frac{p_n}{k}).
\end{equation}
\end{theorem}

\begin{proof}
If there is a $C$ such that for each pair neighbouring prime numbers $p_{n-1}, p_n> C$ satisfy: 
\begin{equation}
p_n - p_{n -1} < \min(\frac{p_{n-1}}{k - 1}, \frac{p_n}{k}).  
\end{equation}
Then theorems \ref{theoremAFT}, \ref{theoremAprimeFT} will be satisfied and the fractional theorem will also be true. 
\end{proof}

\begin{theorem}\label{theoremFFT}
There exists an integer $C$ that that for every pair of neighbouring primes $p_{n-1}, p_n> C$ the following inequality takes place
\begin{equation}
p_n - p_{n -1} < \min(\frac{p_{n-1}}{k - 1}, \frac{p_n}{k}).
\end{equation}
\end{theorem}

\begin{proof}
Since $\frac{p_{n -1}}{k} < \min(\frac{p_{n-1}}{k - }, \frac{p_n}{k})$ we will estimate $C$ using $p_n - p_{n -1} < \frac{p_{n-1}}{k}$.
\newline
\newline
Indeed using one of Proposition 1.10 from \cite{Dusart:1998aa} (For  $k > 463$ , $p_{k+1} \leq p_k(1 + \frac{1}{\ln^2p_k}) )$
we have the following inequality:
\begin{equation}
p_n - p_{n -1} \leq \frac{p_{n -1} }{\ln^2 p_{n -1}} 
\end{equation}
Thus our problem is to find such $n_0$ that for all $n > n_0 + 1$ the following holds:
\begin{equation} 
\frac{p_{n -1}}{\ln^2 p_{n -1}} <   \frac{p_{n -1}}{k} 
\end{equation}
Since $\ln(p_n)$ is a strictly increasing function so there exist such $n_0$ that for any $n > n_0 + 1$ this inequality takes place. Thus we have the following estimate for $C(k)$:
\begin{equation} 
C(k) = \max(p_r, p_{465}), \text{ where } p_{r-1} < \exp(\sqrt{k}) < p_r \text{ and } n_0 = max(r + 1, 465).
\end{equation}
\end{proof}

\begin{theorem}\label{theoremGFT}
The fractional theorem is true for all integers $n > C(k) = max(p_r, p_{465})$, where $p_{r-1} < \exp(\sqrt{k}) < p_r$
\end{theorem}

\begin{proof}
Theorem \ref{theoremGFT} is true due to theorems \ref{theoremEFT}, \ref{theoremFFT}.
\end{proof}

\subsubsection{Application of the fractional theorem}

\begin{theorem}[Strong theorem (exponent 3)]\label{theorem1AFT}
There exists a computable integer $C$ such that for each integer $n>C$ there are at the least two primes $p, q$ such that $(n - 1)^3 < p, q < n^3$.
\end{theorem}

\begin{proof}

\begin{lemma}\label{lemmaAAFT}
Let $k = \frac{g^\frac{3}{2}}{g^\frac{3}{2} - (g - 1)^\frac{3}{2} }$ where $k > 2$ and $g$ is an integer, then there exists a computable integer $C(g)$ such that for each integer $n > C(g)$:
\begin{equation}
(\frac{g - 1}{g})^\frac{3}{2} n < p, q < (\frac{g}{g - 1})^\frac{3}{2} n
\end{equation}
\end{lemma}

\begin{proof}
Lemma \ref{lemmaAFT} is true due to theorem \ref{theoremGFT}.
\end{proof}

\noindent According to the paper \cite{Baker:2001aa} there exist $n_0, x_0$ such that for $n > n_0 + 1$ it would follow $p_{n-1}, p_n > x_0$ and the following inequality holds:
\begin{equation} 
p_n - p_{n -1} < p_n^\frac{21}{40}
\end{equation}

\begin{lemma}\label{lemmaBFT} 
Let $g > x_0$ then there exists such an integer $C(g)$ such that for all $p_{n-1}, p_n > C(g)$ we have:
\begin{equation}
p_n - p_{n -1} < \frac{p_{n-1}}{{(\frac{g^\frac{3}{2}}{(g^\frac{3}{2} - (g - 1)^\frac{3}{2}})}}
\end{equation}
\end{lemma}

\begin{proof}
Our goal to give an estimate of $C(g)$. Since $n > n_0 + 1$, then 
\begin{equation}
p_n - p_{n -1} < p_n^\frac{21}{40}
\end{equation}
Our problem is to find $C(g)$ when for each $p_{n-1}, p_n > C(g)$ the following inequality is satisfied:
\begin{equation}
p_n^\frac{21}{40} < \frac{p_{n -1}}{(\frac{g^\frac{3}{2}}{(g^\frac{3}{2}  - (g - 1)^\frac{3}{2}})} 
\end{equation}
or 			
\begin{equation}
\frac{g^\frac{3}{2}}{g^\frac{3}{2}  - (g - 1)^\frac{3}{2}} < (\frac{p_{n -1}}{{p_n}})^\frac{21}{40} p_{n -1}^\frac{19}{40}
\end{equation}
and finally:
\begin{equation}
(\frac{p_n}{p_{n - 1}})^\frac{21}{19} (\frac{g^\frac{3}{2}}{g^\frac{3}{2}  - (g - 1)^\frac{3}{2}})^\frac{40}{19} < 2^\frac{21}{19} (\frac{g^\frac{3}{2}}{g^\frac{3}{2}  - (g - 1)^\frac{3}{2}})^\frac{40}{19} <  p_{n -1}
\end{equation}

\noindent Then $2^{\frac{21}{19}} (\frac{g^\frac{3}{2}}{(g^\frac{3}{2}  - (g - 1)^\frac{3}{2}})^\frac{40}{19} < 3g^\frac{40}{19}$ and we can take $C(g) = 3g^2([g^\frac{2}{19}] +1)$
\end{proof}

\begin{lemma}\label{lemmaCFT}
The interval $(C(g), (g(g - 1))^\frac{3}{2})$ contains at least two prime numbers.
\end{lemma}

\begin{proof}
According to Bertrand's postulate the interval $(C(g), 2C(g))$ contains at least two or more prime numbers \cite{Sierpinski:1964aa}, \cite{Ramanujan:1919aa}. Let us show that $2C(g)<(g(g - 1))^\frac{3}{2}$. Indeed, $2C(g) < 6(g^\frac{40}{19} + g^2) < (g(g - 1))^\frac{3}{2}$ already takes place when $g > 20$. 
\end{proof}

\noindent Let us take $C$ equal to $p_r$ where $p_{r-1} \leq C(g) < p_r$. Since for any integer $n > p_r$ the fractional theorem is true and as $(g(g - 1))^\frac{3}{2} > p_r$. We can choose $n_0$ such that $| {n_0} - (g(g - 1))^\frac{3}{2}| \leq \frac{1}{2}$. Thus $n_0 = (g(g - 1))^\frac{3}{2} + \theta$, where $| \theta | \leq \frac{1}{2}$. And 
\begin{equation}
(\frac{g - 1}{g})^\frac{3}{2} n_0 = (g - 1)^3 + \theta(\frac{g - 1}{g})^\frac{3}{2} < p, q < (\frac{g}{g - 1})^\frac{3}{2} n_0 = g^3 + \theta (\frac{g}{g - 1})^\frac{3}{2}
\end{equation}
\noindent We have $(g - 1)^3 < p, q < g^3$ since the intervals $((g - 1)^3, (g - 1)^3 + |\theta| \frac{g - 1}{ g}^\frac{3}{2}), (g^3, g^3 + |\theta|\frac{g}{g - 1}^\frac{3}{2})$ don't contain any integers, 
\end{proof}

\subsection{Brocard's conjecture}

\begin{conjecture}[Brocard's conjecture] 
For each pair of neighbouring primes $p_{n-1}, p_n$ there are at the least four prime numbers $p, q, r, s$ where $p_{n-1}^2 < p, q, r, s < p_n^2$.
\end{conjecture}

\begin{theorem}[Brocard's theorem  (exponent 3)]\label{theoremABC}
There exists a computable integer $C(B)$ such that for each pair of neighbouring primes $p_{n-1}, p_n > C(B)$ there are at the least four primes $p, q, r, s$ such that $p_{n-1}^3 < p, q, r, s < p_n^3$.
\end{theorem}

\begin{proof}
Since any interval $(p_{n-1}, p_n)$ can represent as a union $(p_{n-1}, p_n) = \cup_k(p_{n-1}+ k, p_{n-1} + k + 1)$ where k runs from $0$ to $p_n - p_{n-1} - 1$. Thus $(p_{n-1}^3, p_n^3) = \cup_k((p_{n-1}+ k)^3, (p_{n-1}+ k +1)^3)$. The minimal number of such intervals is equal to two so any interval $(p_{n-1}^3, p_n^3)$ contains at the least four prime numbers according to the Strong theorem (exponent 3).
\end{proof}

\begin{theorem}[Strong Brocard's theorem (exponent 3)]
For any natural k there exists such a constant $C(k)$ that for each prime $p_{n-1} > C(k)$ an interval $(p_{n-1}^3, p_n^3)$ contains at the least $2k$ prime numbers. 
\end{theorem}

\begin{proof}
According to the lemma \ref{lemma4}: There exists a computable integer $C$ such that with increasing $n, n > C$ the number of the prime numbers in an interval $((n - 1)^3, n^3)$ grows at the least as $n^{0.425}$. Let us take $n_0 > C$ and $n_0^{0.425} > k$ thus $C(k) = n_0 = [\max(C,k^\frac{40}{17})] + 1$. Let $p_{r-1}\leq n_0 < p_r$ then for any $m \geq r + 1$ the interval $(p_{m-1}^3, p_m^3)$ contains at the least $2k$ prime numbers. 
\end{proof}
 
\noindent Although today we don't have any instruments for proving Brocard's theorem exponent 2, nevertheless offered approach without using Legendre's conjecture permits us to get the following result:

\begin{theorem}[Weak Brocard's theorem (exponent 2)] 
There exists a computable integer $C(B)$ such that for each pair of neighbouring primes $p_{n-1}, p_n > C(B)$ where $(p_n- p_{n-1})>3p_{n-1}^\frac{1}{20}$ there are at the least two prime numbers $p, q$ with $p_{n-1}^2 < p, q < p_n^2$.
\end{theorem}

\begin{proof}

\begin{lemma}\label{lemmaABC}
Let $k = \frac{p_m}{p_m - p_{m-1}}$ where $k > 2$ then it exists a computable integer $C(m)$ that for each integer $n > C(m)$ takes place:
\begin{equation}
\frac{p_{m-1}}{p_m} n < p, q < \frac{p_m}{p_{m-1}}n
\end{equation}
\end{lemma}

\begin{proof}
Lemma \ref{lemmaABC} is true due to theorem 45.
\end{proof}

\noindent According to the paper \cite{Baker:2001aa} there exist $n_0, x_0$ such that for $n > n_0 + 1$ it would follow $p_{n-1}, p_n > x_0$ and the following inequality takes place:
\begin{equation} 
p_n - p_{n -1} < p_n^\frac{21}{40}.
\end{equation}

\begin{lemma}\label{lemmaBBC}
Let $m, n > n_0 + 1$ then there exists such an integer $C(m)$ that for each $p_{n-1} > C(m)$, $p_n - p_{n -1} < \frac{p_{n -1}}{\frac{p_m}{p_m - p_{m-1}}}$
\end{lemma}

\begin{proof}
Our goal is to give an estimate of $C(m)$. Since $n > n_0$, we have that:
\begin{equation} 
p_n - p_{n -1} < p_n^\frac{21}{40}
\end{equation}

\noindent Our problem is to find $C(m)$ for $p_n > C(m)$ when the inequality holds:
\begin{equation}
p_n^\frac{21}{40} < \frac{p_{n -1}}{\frac{p_m}{p_m - p_{m-1}}}
\end{equation}
or
\begin{equation}
\frac{p_m}{p_m - p_{m-1}} < (\frac{p_{n -1}}{p_n})^\frac{21}{40} p_{n -1}^\frac{19}{40}
\end{equation}	  
and finally 			
\begin{equation}
(\frac{p_n}{p_{n - 1}})^\frac{21}{19} (\frac{p_m}{p_m - p_{m-1}})^\frac{40}{19}< p_{n -1}.
\end{equation}
Using the inequality $p_n <2 p_{n - 1}$ we have the following estimate for $C(m)$:
\begin{equation}
C(m) = [2^\frac{21}{19}(\frac{p_m}{p_m - p_{m-1}})^\frac{40}{19}] + 1.
\end{equation}

\end{proof}

\begin{lemma}\label{lemmaCBC}
Let $p_m - p_{m-1} > 3p_{m-1}^\frac{1}{20}$ then the interval $(C(m), p_{m-1}p_m)$ contains at least two primes.
\end{lemma}

\begin{proof}
According to Bertrand's postulate the interval $(C(m), 2C(m))$ contains at least two or more primes \cite{Sierpinski:1964aa}, \cite{Ramanujan:1919aa}. Let us show that under the assumption that $(p_m - p_{m-1}) > 3 p_{m-1}^\frac{1}{20}$ the following inequality takes place:
\begin{equation}
2C(m) < 2^\frac{40}{19}(\frac{p_m}{p_m - p_{m-1}})^\frac{40}{19} + 2 < p_{m-1} p_m
\end{equation}
Indeed if $p_m - p_{m-1} > 3 p_{m-1}^\frac{1}{20}$ we have
\begin{equation}
\frac{2C(m)}{p_{m-1} p_m}<(\frac{p_m}{p_{m-1}})^\frac{21}{19}(\frac{2p_{m-1}^\frac{1}{20}}{{p_m - p_{m-1}}})^\frac{40}{19}+\frac{2}{p_{m-1} p_m} < 2^\frac{21}{19}(\frac{2}{3})^\frac{40}{19}+ 0.02 < 1.
\end{equation}
Thus the interval $(C(m), p_{m-1}p_m)$ contains at the least more than two primes.
\end{proof}

\begin{lemma}\label{lemmaDBC}
There exists such an integer $C$ such that for each $n > C$ the following holds:
\begin{equation}
\frac{p_{m-1}}{p_m}n < p, q< \frac{p_m}{p_{m-1}}n
\end{equation}
\end{lemma}

\begin{proof}
Let us take $C = p_r$ where $p_{r-1} \leq C(m) < p_r$ , $C(m) = [2^\frac{21}{19} (\frac{p_m}{p_m - p_{m-1}})^\frac{40}{19}] + 1$. 
\newline
\newline
Since the inequality $\frac{p_{m-1}}{p_m}n < p, q < \frac{p_m}{p_{m-1}}n$ takes place for all primes $p_n > p_r$ so according to theorems \ref{theoremEFT}, \ref{theoremFFT} this inequality is true for all natural $n \geq p_{r+1} > p_r$.
\end{proof}
\noindent Thus we can take $n$ as equal to $p_{m-1}p_m$ and then we would have that the interval $(p_{m-1}^2, p_m^2)$ contains at least two primes under the condition that the difference between consecutive primes satisfies the inequality $p_m - p_{m-1} > 3p_{m-1}^\frac{1}{20}$.
\end{proof}

\bibliography{bibiliography}
\bibliographystyle{plain}

\end{document}